\documentclass[11pt]{article}
\usepackage[utf8]{inputenc}
\usepackage[T1]{fontenc}
\usepackage[english]{babel}
\usepackage{csquotes}

\usepackage{amsmath}
\usepackage{amsfonts}
\usepackage{amssymb}
\usepackage{stix2}

\DeclareSymbolFont{bbold}{U}{bbold}{m}{n}
\DeclareSymbolFontAlphabet{\mathbbold}{bbold}
\DeclareFontFamily{U}{min}{}
\DeclareFontShape{U}{min}{m}{n}{<-> udmj30}{}

\usepackage{extarrows}

\usepackage{stmaryrd}
\usepackage{array}
\usepackage{tikz,tkz-tab}
\usepackage{tikz-cd}
\usetikzlibrary{decorations,arrows,shapes,matrix,calc}
\usepackage{bbm}
\usepackage[all]{xy}

\usepackage{titlesec}
\usepackage{verbatim}
\usepackage{enumitem}
\usepackage{geometry}
\usepackage{amsthm}
\usepackage{fancyvrb}
\usepackage{todonotes}

\newtheoremstyle{theoremdd}
{\topsep}
{\topsep}
{\itshape}
{0pt}
{\bfseries}
{.}
{ }
{\thmname{#1}\thmnumber{ #2}\textnormal{\thmnote{ (#3)}}}
\theoremstyle{theoremdd}
\newtheorem{theorem}{Theorem}[section]
\newtheorem{definition}[theorem]{Definition}
\newtheorem{proposition}[theorem]{Proposition}

\newtheorem{corollary}[theorem]{Corollary}
\newtheorem{lemma}[theorem]{Lemma}

\makeatletter
\def\@endtheorem{\qed\endtrivlist\@endpefalse}
\makeatother

\newtheoremstyle{theoremddnoproof}
{\topsep}
{\topsep}
{\itshape}
{0pt}
{\bfseries}
{.}
{ }
{\thmname{#1}\thmnumber{ #2}\textnormal{\thmnote{ (#3)}}}
\theoremstyle{theoremddnoproof}
\newtheorem{theoremnoproof}[theorem]{Theorem}

\newtheorem{corollarynoproof}[theorem]{Corollary}
\newtheorem{lemmanoproof}[theorem]{Lemma}

\usepackage{hyperref}
\usepackage{cleveref}
\usepackage[style=alphabetic, backend=biber, maxbibnames=99]{biblatex}
\usepackage{colonequals}

\newcommand{\KM}{\ensuremath{\mathrm{K}^M}}
\newcommand{\KMW}{\ensuremath{\mathrm{K}^{MW}}}
\newcommand{\KW}{\ensuremath{\mathrm{K}^W}}
\newcommand{\GW}{\ensuremath{\mathrm{GW}}}

\title{Milnor-Witt K-theory and Witt K-theory of a field}

\author{Robin Carlier\footnote{The author is supported
		by the ANR project ``HQDIAG'', ANR-21-CE40-0015.}}
\date{\today}
\begin{document}

\maketitle
\begin{abstract}
	We recall some basic computations in the Milnor-Witt K-theory of a field, following Morel. We then focus on the Witt K-theory of a field of characteristic two and give an elementary proof of the fact that it is isomorphic as a graded ring to the Rees algebra of the fundamental ideal of the Witt ring of symmetric bilinear forms using Kato's solution to Milnor's conjecture on quadratic form.
\end{abstract}
\section{Introduction}
These notes are intended as a companion to~\cite{deglise_KMW_2023}.
While the relevance of Milnor-Witt K-theory to $\mathbb{A}^1$-homotopy theory is explained there, we take a much more elementary approach here and focus on the case of fields.
In the first section, we recall the basics of Milnor-Witt K-theory following the definition of Hopkins and Morel~\cite[Def. 5.1]{morel_puissances2004} and recall some basics computations in the ring $\KMW_*(F)$.
The main result in this first section is Theorem~\ref{prop_iso}, which fully describes the negative part of Milnor-Witt K-theory as well as Corollary~\ref{cor_loc}, which is a direct consequence of Theorem~\ref{prop_iso} and which describes the $\eta$-periodic structure of Milnor-Witt K-Theory.
Our exposition in the first section follows extremely closely section 1 of chapter 3 of~\cite{morel_2012}, and we merely gives a few more details in some proofs for the convenience of the reader.

In a second part, we focus on Witt K-theory, which is the quotient of Milnor-Witt K-theory by the hyperbolic element $h$, mirroring the definition of the Witt Ring of $F$ as the quotient of the Grothendieck-Witt ring of $F$ by the element representing the hyperbolic inner product space.
We mainly focus on the characteristic 2 situation.
The main result of this second part is Theorem~\ref{KW_iso_I} asserting that the Witt $K$-theory of $F$ is in fact isomorphic to the $\mathbb{Z}$-graded ring associated to the $\mathrm{I}(F)$-adic filtration of the Witt ring $\mathrm{W}(F)$.
The theorem also holds in the case of an arbitrary field as stated in Theorem 2.3.5 of~\cite{deglise_KMW_2023} and a proof in characteristic $\neq 2$ case appears as the main result of~\cite{morel_puissances2004} and here we merely check that the proof strategy from there can be adapted.
In fact, the characteristic 2 case happens to be much more elementary than the general case, the main input being Kato's proof of Milnor's conjecture on $\KM_*(F)/2\KM_*(F)$ in characateristic $2$~\cite{kato_symmetric_1982}, which is rather elementary compared to the general case.
Validity of~\cite[Thm. 2.3.5]{deglise_KMW_2023} in characteristic 2 has been claimed in various places (\cite[Remark 3.12]{morel_2012}, for instance) and is surely known by the experts.
As explained in~\cite[Cor. 2.3.7]{deglise_KMW_2023}, a consequence of the isomorphism of Theorem~\ref{KW_iso_I} is the existence of a cartesian square
\[\begin{tikzcd}
		\KMW_*(F) & \KM_*(F) \\
		\mathrm{I}^*(F) & i^*(F)
		\arrow["{}", from=1-1, to=1-2]
		\arrow["{}"', from=1-1, to=2-1]
		\arrow["{}", from=1-2, to=2-2]
		\arrow["{}"', from=2-1, to=2-2]
		\arrow["\ulcorner"{anchor=center, pos=0.125}, draw=none, from=1-1, to=2-2]
	\end{tikzcd}\]
where $i^*(F) = \bigoplus\limits_{n}\mathrm{I}^n(F)/\mathrm{I}^{n+1}(F)$ and $\mathrm{I}^{n}(F)$ is the $n$-th power of the fundamental ideal $\mathrm{I}(F)$ if $n > 0$ and $\mathrm{W}(F)$ if $n \leq 0$.

These notes grew out of notes for a seminar on Morel's book~\cite{morel_2012}.
We wish to thank Frédéric Déglise for prompting us to write them and for suggesting to work out the characteristic 2 case.
The results and computations described in these notes are currently being formalized using the Lean proof assistant, and the project should soon appear on the author's personal webpage.
\section{Basics of Milnor-Witt K-theory of fields, after Morel}
\begin{definition}
	Let $F$ be a field. The {\normalfont{Milnor-Witt $K$-theory of $F$}}, denoted by $\KMW_*(F)$, is the associative graded ring presented by generators and relations in the following way:\begin{itemize}
		\item A set of generator is given by the set of symbols $\left\{[u],\, u \in F^\times\right\} \cup \{\eta\}$.
		      The symbol $\eta$ is of degree $-1$, and for $u \in F^\times$ the symbol $[u]$ is of degree $1$.
		\item The generators are subject to the following set of relations: \begin{enumerate}[label={\normalfont{$\mathbf{KMW\arabic*}$}}]
			      \item\label{rel1} For all $a \in F^{\times}$, $a \neq 1$, $[a][1-a] = 0$.
			      \item\label{rel2} For all $a, b \in F^{\times}$, $[ab] = [a] + [b] + \eta \cdot [a] [b]$.
			      \item\label{rel3} For all $u \in F^\times$, $\eta[u] = [u]\eta$.
			      \item\label{rel4} Let $h \colonequals \eta[-1] + 2$. Then $\eta h = 0$.
		      \end{enumerate}
	\end{itemize}
\end{definition}
Relation~\ref{rel1} is usually called the \textit{Steinberg relation}.
It is immediate from the definition that the quotient $\KMW_*(F)/\eta$ is isomorphic to the Milnor $K$-theory $\KM_*(F)$ of the field $F$.
One defines the Witt $K$-theory of $F$, denoted by $\KW_*(F)$, as the quotient $\KMW_*(F)/h$.
It follows immediately from relation~\ref{rel3} that $\eta$ is a central element in $\KMW_*(F)$, so that one gets the following easy presentation of the $n$-the graded piece of $\KMW_*(F)$.
\begin{lemmanoproof}\label{lemma_generators_big}
	Let $F$ be a field and let $n$ be an integer. Let $\widetilde{\mathrm{K}}^{MW}_n(F)$ be the abelian group presented by generators and relations in the following way: \begin{itemize}
		\item A set of generators is the set of symbols $\left[\eta^m, u_1,\ldots, u_r\right]$ where $r$ is an integer, $m = r - n \geq 0$, and $u_1,\ldots,u_r$ are elements of $F^\times$.
		\item These generators satisfy the following set of relations: \begin{itemize}
			      \item $\left[\eta^m, u_1,\ldots, u_r\right] = 0$ if there exists $i$ such that $1 \leq i < r$ and $u_i + u_{i+1} = 0$.
			      \item For all integers $r$, for all pairs $a,b \in F^\times$ and for all $1 < i < r $,\begin{align*}
				            \left[\eta^m,\ldots,u_{i-1},ab,u_{i+1},\ldots,u_r\right] & = \left[\eta^m,\ldots,u_{i-1},a,u_{i+1},\ldots,u_r\right]                      \\
				                                                                     & \qquad \qquad + \left[\eta^m,\ldots,u_{i-1},b,u_{i+1},\ldots,u_r\right]        \\
				                                                                     & \qquad \qquad + \left[\eta^{m+1},\ldots,u_{i-1},a,b,u_{i+1},\ldots,u_r\right].
			            \end{align*}
			      \item For all integers $i$, \begin{align*}
				            \left[\eta^{m},u_1,\ldots,u_{i-1},-1,u_{i+1},\ldots,u_r\right] + 2\left[\eta^{m-1},u_1,\ldots,u_{i-1},u_{i+1},\ldots,u_r\right] = 0.
			            \end{align*}
		      \end{itemize}
	\end{itemize}
	The canonical morphism $\widetilde{\mathrm{K}}^{MW}_n(F) \to \KMW_n(F)$ that sends $[\eta^m, u_1,\ldots, u_r]$ to $\eta^m[u_1]\cdots[u_r]$ is an isomorphism.
\end{lemmanoproof}
For an element $a \in F^\times$, we set $\langle a \rangle \colonequals 1 + \eta[a] \in \KMW_0(F)$.
We now give a few relations that hold in the ring $\KMW_*(F)$. The proofs are extracted from~\cite{morel_2012}, chapter 3.1 and are reproduced here for convenience.
\begin{lemma}\label{lemma_1}
	Let $a,b \in F^\times$.
	\begin{enumerate}[label={\normalfont{({\alph*})}}, ref={\normalfont{(\alph*)}}]
		\item\label{lemma_1_1} $[ab] = [a] + \langle a\rangle[b]$ and $[ab] = [a]\langle b\rangle + [b]$.
		\item\label{lemma_1_2} $\langle ab\rangle = \langle a\rangle.\langle b\rangle$ and $\langle a \rangle $ is central in $\KMW_*(F)$.
		\item\label{lemma_1_3} $\langle1_{F}\rangle = 1_{\KMW_*(F)}$ and $[1] = 0$.
		\item\label{lemma_1_4} $\langle a\rangle$ is a unit of $\KMW_*(F)$ whose inverse is $\langle a^{-1}\rangle$.
		\item\label{lemma_1_5} $\left[\frac{a}{b}\right] = [a] - \left\langle\frac{a}{b}\right\rangle[b]$. In particular $\left[a^{-1}\right] = -\left\langle a^{-1} \right\rangle[a]$.
	\end{enumerate}
\end{lemma}
\begin{proof}
	One has \begin{align*}
		[ab] & = [a] + [b] + \eta[a][b]       \\
		     & = [a] + (1 + \eta[a])[b]       \\
		     & = [a] + \langle a \rangle [b],
	\end{align*}
	with the first equality being relation~\ref{rel2}, and the last being the definition of $\langle a \rangle$.
	Since $[ab] = [ba]$, one also gets $[ab] = [a] + [b]\langle a\rangle$ and $[ab] = [b] + \langle b \rangle[a]$, which proves the first item.

	For the proof of~\ref{lemma_1_2}, using again relation~\ref{rel2}, we have \begin{align*}
		\langle ab \rangle & = 1 + \eta[ab]                                \\
		                   & = 1 + \eta\left([a] + [b] + \eta[a][b]\right) \\
		                   & = 1 + \eta[a] + \eta[b] + \eta^2[a][b]        \\
		                   & = (1 + \eta[a])(1 + \eta[b])                  \\
		                   & = \langle a \rangle\langle b \rangle.
	\end{align*}
	Using property~\ref{lemma_1_1}, this gives $\langle a \rangle [b] = [b] \langle a \rangle$ for all $b$, thus $\langle a \rangle$ is central as claimed.

	For the third item, notice that $h = \langle -1 \rangle + 1$ and that $\eta[1] = \langle 1 \rangle - 1$. Thus, relation~\ref{rel4} implies
	\[(\langle 1 \rangle - 1)(\langle -1 \rangle + 1) = 0,\]
	and since
	\[ (\langle 1 \rangle - 1)(\langle -1 \rangle + 1) = \langle 1 \rangle \langle -1 \rangle + \langle 1 \rangle - \langle -1 \rangle - 1, \]
	we have $\langle 1 \rangle \langle -1 \rangle = \langle -1 \rangle$ by~\ref{lemma_1_2}, so that
	\[(\langle 1 \rangle - 1)(\langle -1 \rangle + 1) = \langle 1 \rangle - 1.\]
	Thus the fact that the right-hand side is zero implies the desired relation. Now by part~\ref{lemma_1_1}, \begin{align*}
		[1] & = [1] + \langle 1 \rangle [1] \\
		    & = [1] + [1]
	\end{align*}
	so that $[1] = 0$ as desired.\\
	For the fourth point, by part~\ref{lemma_1_3} and~\ref{lemma_1_2} we get \begin{align*}
		1 & = \langle 1 \rangle                         \\
		  & = \langle a a^{-1}\rangle                   \\
		  & = \langle a \rangle \langle a^{-1} \rangle,
	\end{align*}
	and similarly $\langle a^{-1} \rangle \langle a \rangle = 1$, which is~\ref{lemma_1_4}.

	Finally, one has
	\begin{align*}
		[a] & = \left[ab^{-1}b\right]                                              \\
		    & = \left[\frac{a}{b}\right] + \left\langle\frac{a}{b}\right\rangle[b]
	\end{align*}
	which proves the equality $\left[\frac{a}{b}\right] = [a] - \left\langle\frac{a}{b}\right\rangle[b]$. Setting $a = 1$ in this equality and using that $[1] = 0$, we get the expected particular case.
\end{proof}
\begin{lemma}\label{lemma_generators}
	\begin{enumerate}[label={\normalfont{(\arabic*)}}]
		\item For $n \geq 1$, the abelian group $\KMW_n(F)$ is generated by elements of the form $[u_1]\cdots[u_n]$.
		\item For $n \leq 0$, the abelian group $\KMW_n(F)$ is generated by elements of the form $\eta^{n}\langle u \rangle$.
	\end{enumerate}
\end{lemma}
\begin{proof}
	Assume first that $n \geq 1$. By Lemma~\ref{lemma_generators_big} we already know that $\KMW_n(F)$ is generated by elements of the form $\eta^m[u_1]\cdots[u_r]$ with $n = r-m$.
	Relation~\ref{rel2} can be rephrased as $\eta[a][b] = [ab] - [a] - [b]$, thus, we can inductively decrease $m$ until $m = 0$, leaving only terms of the desired form.
	In the case $n \leq 0$, the group $\KMW_n(F)$ is again generated by elements of the form $\eta^m[u_1]\cdots[u_r]$ with $n = m-r$. Notice first that \begin{align*}
		\eta^m[u_1]\cdots[u_r] = \eta^{n}\eta^{r}[u_1]\cdots[u_r].
	\end{align*}
	By part~\ref{lemma_1_3} of the previous lemma, \begin{align*}
		\eta[a] & = \langle a \rangle - 1                 \\
		        & = \langle a \rangle - \langle 1 \rangle
	\end{align*}
	and this relation as well as relation~\ref{lemma_1_2} of the previous lemma ensure that $\eta^{r}[u_1]\cdots[u_r]$ reduces to a sum of elements of the form $\langle u \rangle$.
\end{proof}
Let $\epsilon$ denotes the element $-\langle-1\rangle$ of $\KMW_0(F)$. Relation~\ref{rel4} can be rephrased as $\epsilon.\eta = \eta$. The following is again extracted from~\cite{morel_2012}, 3.1.
\begin{lemma}\label{lemma_2}
	Let $a, b \in F^\times$, the following relations hold in $\KMW_*(F)$: \begin{enumerate}[label={\normalfont{({\alph*})}}, ref={\normalfont{(\alph*)}}]
		\item\label{lemma_2_1} $[a][-a] = 0$,
		\item\label{lemma_2_2} \begin{align*}
			[a][a] & = [a][-1]          \\
			       & = \epsilon[a][-1]  \\
			       & = [-1][a]          \\
			       & = \epsilon[-1][a],
		\end{align*}
		\item\label{lemma_2_3} $\langle a^2 \rangle = 1$,
		\item\label{lemma_2_4} $\langle a \rangle + \langle -a \rangle = h$,
		\item\label{lemma_2_5} $[a][b]=\epsilon[b][a]$.
	\end{enumerate}
\end{lemma}
\begin{proof}
	Let us prove the first point. If $a = 1$, this is clear since $[1] = 0$. Assume $a \neq 1$.
	From the equality $-a = \frac{1 - a}{1 - a^{-1}}$ we can use relation~\ref{lemma_1_5} of Lemma~\ref{lemma_1} to get \begin{align*}
		[-a] = [1 - a] - \langle -a \rangle [1 - a^{-1}].
	\end{align*}
	Multiplying this equality by $[a]$, we obtain \begin{align*}
		[a][-a] & = [a][1-a] - \langle -a \rangle [a][1-a^{-1}]               \\
		        & = -\langle -a \rangle [a][1-a^{-1}]                         \\
		        & = \langle -a \rangle \langle a \rangle [a^{-1}][1 - a^{-1}] \\
		        & = 0,
	\end{align*}
	the second equality and last equality being the Steinberg relation~\ref{rel1} and the third being point~\ref{lemma_1_5} of~\ref{lemma_1}.

	For the proof of~\ref{lemma_2_2}, notice first that by part~\ref{lemma_1_1} of Lemma~\ref{lemma_1}, one has $[-a] = [-1] + \langle -1\rangle [a]$.
	Multiplying this on the left by $[a]$ and using the previously proved relation $[a][-a] = 0$, one obtains $[a][-1] = - \langle -1 \rangle [a][a]$, that is to say $[a][-1] = \epsilon[a][a]$.
	Similarly, by multiplying $[-a] = [-1] + \langle -1\rangle [a]$ on the right by $[a]$ gives $[-1][a] = \epsilon[a][a]$.
	Noticing that $\epsilon^2 = 1$, this proves that $[a][a] = \epsilon[a][-1]$ and that $[a][a] = \epsilon[-1][-a]$. Since one has \begin{align*}
		0 & = [1]                             \\
		  & = \left[{(-1)}^2\right]           \\
		  & = [-1] + \langle -1 \rangle [-1],
	\end{align*}
	one gets that $\epsilon[-1] = [-1]$ so that $\epsilon[a][-1] = [a][-1]$ and $\epsilon[-1][-a] = [-1][a]$, which finishes proving~\ref{lemma_2_2}.

	Let us prove part~\ref{lemma_2_3}. By definition of $\left \langle a^2 \right \rangle$ it is enough to prove that $\eta[a^2] = 0$.
	As
	\[
		\left[a^2\right] = 2[a] + \eta[a][a],
	\]
	one obtains $[a][a] = [-1][a]$ by~\ref{lemma_2_2} so that $\left[a^2\right] = 2[a] + \eta[-1][a]$, \textit{i.e} $\left[a^2\right] = (2 + \eta[-1])[a]$.
	Multiplying by $\eta$ and using relation~\ref{rel4} yields $\eta\left[a^2\right] = 0$.
	We now prove~\ref{lemma_2_4}.
	One starts with $[a][-a] = 0$ and multiplies by $\eta^2$ to get $\eta[a]\eta[-a] = 0$.
	Since $\eta[a] =\langle a \rangle - 1$ and $\eta[-a] =\langle -a \rangle - 1$, by expanding $(\langle a \rangle - 1)(\langle -a \rangle - 1)$ one gets $\langle a \rangle + \langle -a\rangle = 1 + \left\langle -a^2 \right\rangle$.
	By~\ref{lemma_2_3} along with part~\ref{lemma_1_2} of Lemma~\ref{lemma_1}, we get $\left\langle -a^2\right\rangle = \langle -1 \rangle$ thus $\langle a \rangle + \langle -a\rangle = h$.

	Finally, for the proof of~\ref{lemma_2_5}, one starts with the relation $[ab][-ab] = 0$ which was proved in~\ref{lemma_2_1}. Expanding it using relation~\ref{lemma_1_1} of~\ref{lemma_1} gives
	\[
		([a] + \langle a \rangle [b])([-a] + \langle -a \rangle [b]) = 0.
	\]
	Expanding the left hand side of this equality yields
	\[
		[a][-a] + \langle -a \rangle [a][b] + \langle a \rangle [b][-a] + \langle -a^2 \rangle [b][b] = 0.
	\]
	By part~\ref{lemma_2_1} we have $[a][-a] = 0$, by part~\ref{lemma_2_2} we also have $[b][b] = [-1][b]$,
	and as noted before $\langle -a^2 \rangle = \langle - 1\rangle $. Thus above equality becomes
	\[
		\langle a \rangle\left([b][-a] + \langle -1 \rangle [a][b]\right) + \langle -1 \rangle [-1][b] = 0.
	\]
	Using that $[-a] = [a] + \langle a \rangle [-1]$, this gives
	\[
		\langle a \rangle\left([b][a] + \langle -1 \rangle [a][b]\right) + \langle a^2 \rangle [b][-1]+\langle -1 \rangle [-1][b] = 0.
	\]
	and once again $\langle a^2 \rangle = 1$ and $[b][-1]  +\langle -1 \rangle [-1][b] = 0$ by~\ref{lemma_2_2} (recall that $\epsilon = -\langle-1\rangle$).
	Thus we obtain
	\[
		\langle a \rangle\left([b][a] + \langle -1 \rangle [a][b]\right) = 0.
	\]
	Since $\langle a \rangle$ is a unit by~\ref{lemma_1_4} of~\ref{lemma_1} this implies the equality $[b][a] = \epsilon[a][b]$, which is the desired relation since $\epsilon^2 = 1$.
\end{proof}
We immediately obtain the following
\begin{corollarynoproof}
	The ring $\KMW_*(F)$ is $\epsilon$-graded commutative.
\end{corollarynoproof}
Let us recall some facts about the Grothendieck-Witt ring of a field $F$.
\begin{definition}
	Let $F$ be a field. The {\normalfont{Grothendieck-Witt group of $F$}}, denoted by $\GW(F)$, is the group completion of the monoid of isomorphism classes of symmetric inner product spaces over $F$ with respect to the addition given by orthogonal sum.
\end{definition}
Notice that the Grothendieck-Witt group inherits a commutative ring structure via the tensor product of symmetric inner product spaces.
For a unit element $u$ of $F$, let $\langle u \rangle$ be the inner product space of dimension one with a basis element $e$ such that $\left(e | e\right) = u$.
If we denote by $h$ again the element $1 + \langle -1 \rangle$ of $\GW(F)$, then relation~\ref{GW2} (see Theorem~\ref{GW_relations} below) shows that the subgroup generated by $h$ is actually an ideal of $\GW(F)$.
The Witt ring $\mathrm{W}(F)$ of $F$ is by definition the quotient $\GW(F)/(h)$. We will work with the Grothendieck-Witt group through an explicit presentation:
\begin{theorem}\label{GW_relations}
	As an abelian group, the Grothendieck-Witt Ring of $F$ is generated by the set of symbols $\left\{\langle u\rangle,\,u \in F^\times\right\}$.
	The following relations give a presentation of $\GW(F)$ with this generating set.
	\begin{enumerate}[label={\normalfont{$\mathbf{GW\arabic*}$}}]
		\item\label{GW1} $\langle uv^2\rangle = \langle u\rangle$.
		\item\label{GW2} $\langle u \rangle + \langle -u \rangle = 1 + \langle-1\rangle$.
		\item\label{GW3} $\langle u \rangle + \langle v \rangle = \langle u + v \rangle + \langle (u+v)uv \rangle$ if $u + v \neq 0$.
	\end{enumerate}
\end{theorem}
This proposition is classical, but we give a particular attention to the characteristic 2 case.
\begin{proof}
	In characteristic different from $2$, any symmetric inner product space admits an orthogonal basis,
	thus the claim about generators follow.

	In characteristic $2$, any symmetric inner product space can be written as an orthogonal sum of dimension $1$ spaces and some symplectic space $N$~\cite[Cor.~I.3.3]{milnor_husemoller1973}.
	The symplectic space must admits a symplectic basis, such that the matrix of its inner product is $H_n = \begin{pmatrix}0 & \mathrm{I}_n \\ \mathrm{I}_n & 0\end{pmatrix}$~\cite[Corr.~I.3.5]{milnor_husemoller1973}.
	To prove the claim about generators, it suffices to show that $H_n = 2n\langle 1 \rangle$ in the Grothendieck-Witt group.
	An easy reordering of the basis elements shows that $H_n$ is an orthogonal sum of $n$ hyperbolic planes $H_1$, thus it suffices to show $H_1 = 2$ in $\GW(F)$.
	While these two spaces might not be isomorphic as symmetric inner product spaces, there is an isomorphism of symmetric inner product spaces $3\langle 1 \rangle \cong \langle 1 \rangle + H$.
	Indeed, let $V$ be a symmetric inner product space over a field of characteristic 2 with an orthogonal basis $(e_1,e_2,e_3)$ such that $\left(e_i | e_i\right) = 1$ for all $i$,
	then the basis $(e_1 + e_2 + e_3, e_1 + e_3, e_2 + e_3)$ exhibits $V$ as isomorphic to $\langle 1 \rangle + H_1$.

	All three relations are easily seen to be satisfied in $\GW(F)$ and the harder part is to show that any relation arises from the ones given in the statement of the proposition.
	Assume \[\langle \alpha_1 \rangle + \langle \alpha_2 \rangle = \langle \beta_1 \rangle + \langle \beta_2 \rangle\] in $\GW(F)$.
	Recall from~\cite[Thm.~III.1.7]{milnor_husemoller1973} that a class in the Witt ring contains one and only one anisotropic inner product space up to isomorphism,
	and that any inner product space splits as an orthogonal sum of a metabolic\footnote{Recall that a symmetric inner product space $F$ is said to be metabolic if there exists a subspace $N \subset F$ such that $N^\bot = N$.} space and an anisotropic one.
	Thus, if $\langle\alpha_1\rangle + \langle\alpha_2\rangle$ contains an isotropic vector, then so does $\langle \beta_1 \rangle + \langle \beta_2 \rangle$.
	In this case, there are some $x_1$ and $x_2$ in $F$ such that $\alpha_1 x_1^2 + \alpha_2 x_2^2 = 0$. This means that $\langle \alpha_2 \rangle = \langle -\alpha_1\rangle$,
	so using~\ref{GW1} we may replace the couple $(\alpha_1, \alpha_2)$ by $(\alpha_1, -\alpha_1)$, and hence replace it by $(1, -1)$, and then do the same with $(\beta_1,\beta_2)$.
	If the spaces are anisotropic, then they are isomorphic, then $\beta_1 = \alpha_1x_1^2 + \alpha_2x_2^2$ for some $x_1, x_2$, so that we may replace $(\alpha_1, \alpha_2)$ by $(\alpha_1x_1^2, \alpha_2x_2^2)$ using~\ref{GW1}, and then replace it by $(\beta_1, \beta_1\alpha_1\alpha_2)$ using~\ref{GW3}.
	But the spaces are isomorphic, so their determinant modulo ${\left(F^{\times}\right)}^2$ must be the same and hence $\beta_2 = \beta_1\alpha_1\alpha_2c^2$ for some $c \in F^\times$. Thus~\ref{GW1} again, we can replace $(\beta_1, \beta_1\alpha_1\alpha_2)$ by $(\beta_1,\beta_2)$.

	If there is an equality \[
		\sum\limits_{i=1}^n\langle\alpha_i\rangle = \sum\limits_{j=0}^{m}\langle \beta_i \rangle
	\]
	in $\GW(F)$ then $n = m$ by considering the rank.
	One then concludes using Lemma~\ref{switch_in_gw} below that appears as Lemma III.5.6 in~\cite{milnor_husemoller1973} which is valid in characteristic 2.
	The lemma there is stated for $\mathrm{W}(F)$, but its proof also works \textit{mutatis mutandis} for $\GW(F)$.
\end{proof}
\begin{lemmanoproof}\label{switch_in_gw}
	Let $\alpha_1,\ldots,\alpha_n$ and $\beta_1,\ldots,\beta_n$ be elements of $F^\times$ such that $\sum\limits_{i=1}^n\langle \alpha_i \rangle = \sum\limits_{i=1}^n\langle \beta_i \rangle$ in $\GW(F)$. It is possible to rewrite $(\alpha_1,\ldots,\alpha_n)$ into $(\beta_1,\ldots,\beta_n)$ by changing only two entries at a time while preserving the Grothendick-Witt class.
\end{lemmanoproof}
The Witt ring $\mathrm{W}(F)$ has a unique prime ideal $\mathrm{I}(F)$ such that $\mathrm{W}(F)/\mathrm{I}(F) = \mathbb{Z}/2\mathbb{Z}$.
This ideal is the kernel of the rank modulo $2$ morphism.
Through this rank modulo $2$ morphism, one gets a canonical cartesian square
\[\begin{tikzcd}
		\GW(F) & \mathbb{Z} \\
		\mathrm{W}(F) & \mathbb{Z}/2\mathbb{Z}
		\arrow["{\mathrm{rk}}", from=1-1, to=1-2]
		\arrow["{}"', from=1-1, to=2-1]
		\arrow["{}", from=1-2, to=2-2]
		\arrow["{\mathrm{rk}\ \mathrm{mod}\ 2}"', from=2-1, to=2-2]
		\arrow["\ulcorner"{anchor=center, pos=0.125}, draw=none, from=1-1, to=2-2].
	\end{tikzcd}\]
\begin{lemma}\label{GW_to_KMW} The subgroup of $\KMW_0(F)$ generated by the elements $\langle u\rangle$ satisfies the relations of $\GW(F)$.
\end{lemma}
\begin{proof}
	Parts~\ref{lemma_2_3} and~\ref{lemma_2_4} of Lemma~\ref{lemma_2} show that relations~\ref{GW1} and~\ref{GW2} are satisfied.
	Let us show that~\ref{GW3} is satisfied.
	Since the symbol $\langle u \rangle$ is multiplicative, one may assume that $u+v = 1$.
	By applying~\ref{rel1}, one has $[u][v] = 0$.
	Multiplying this relation by $\eta^2$, using equalities $\eta[u] =\langle u \rangle - 1$ and $\eta[v] =\langle v \rangle - 1$ and expanding, one gets
	\[
		\langle u \rangle + \langle v \rangle = 1 + \langle uv \rangle
	\]
	as expected.
\end{proof}

This lemma implies that there is a well-defined morphism $\phi_0: \GW(F) \to \KMW_0(F)$ which is moreover surjective by Lemma~\ref{lemma_generators}.
By relation~\ref{rel4}, multiplication by $\eta$ kills $h$. Thus, for any $n > 0$, the composition of surjective morphisms
\[
	\GW(F) \xrightarrow{} \KMW_0(F) \xrightarrow{\cdot \eta^n} \KMW_{-n}(F)
\]
factors to a surjective map $\phi_{-n}: \mathrm{W}(F) \to \KMW_{-n}(F)$.
\begin{theorem}\label{prop_iso}
	For all $n \geq 0$, the map $\phi_{-n}$ defined above is an isomorphism.
\end{theorem}
\begin{proof}
	Recall Milnor's epimorphism $s_n: \KM_n(F) \to i^n(F)$ where
	$i^n(F) = \mathrm{I}^{n}(F)/\mathrm{I}^{n+1}(F)$ is the $n$-th graded piece of the graded ring associated to the filtration $\mathrm{I}^n(F)$ of $\mathrm{W}(F)$.
	By definition, $s_n$ sends the symbol $\{a_1,\ldots, a_n\}$ to the class of the symmetric inner product space $\langle1, -a_1\rangle \otimes \cdots \otimes \langle 1, -a_n \rangle$.
	For $n \leq 0$, we set $\mathrm{I}^n(F) = \mathrm{W}(F)$ and for all $n \in \mathbb{Z}$ we let $\mathrm{J}^n(F)$ be the abelian group defined a cartesian square
	\[\begin{tikzcd}
			\mathrm{J}^n(F) & \mathrm{I}^n(F) \\
			\KM_n(F) & i^n(F)
			\arrow["{}", from=1-1, to=1-2]
			\arrow["{}"', from=1-1, to=2-1]
			\arrow["{}", from=1-2, to=2-2]
			\arrow["{}"', from=2-1, to=2-2]
			\arrow["\ulcorner"{anchor=center, pos=0.125}, draw=none, from=1-1, to=2-2]
		\end{tikzcd}\]
	in which the rightmost vertical map is the quotient map $\mathrm{I}^n(F) \twoheadrightarrow \mathrm{I}^n(F)/\mathrm{I}^{n+1}(F)$.
	The collection of modules $\mathrm{J}^{*}(F)$ form a graded ring as the maps in the pullback defining $\mathrm{J}^n$ respect multiplication as $n$ varies.
	Notice that for $n = -1$, this pullback square is actually the square
	\[\begin{tikzcd}
			\mathrm{W}(F) & \mathrm{W}(F) \\
			0 & 0
			\arrow["{}", from=1-1, to=1-2]
			\arrow["{}"', from=1-1, to=2-1]
			\arrow["{}", from=1-2, to=2-2]
			\arrow["{}"', from=2-1, to=2-2]
			\arrow["\ulcorner"{anchor=center, pos=0.125}, draw=none, from=1-1, to=2-2]
		\end{tikzcd}\]
	so that there is a canonical isomorphism \[\mathrm{J}^{-1}(F) \cong \mathrm{W}(F).\]
	Let $\eta \in \mathrm{J}^{-1}(F)$ be the element that corresponds to the element $1 \in \mathrm{W}(F)$ through this isomorphism.
	Let $u \in F^\times$ and let $[u]$ be the element of $\mathrm{J}^{1}(F)$ corresponding to the pair $(\{u\}, \langle u \rangle - 1) \in \KM_1(F) \times \mathrm{I}(F)$.
	As $u$ is sent to $\langle -1, u \rangle \in i^1(F)$ and since $\langle u \rangle - 1 = \langle u \rangle + \langle - 1 \rangle$ in $\mathrm{W}(F)$
	and as both classes represent a symmetric inner product space with an orthogonal basis $e_1, e_2$ such that $\left(e_1 | e_1\right) = -1$ and $\left(e_2 | e_2\right) = u$, they indeed map to the same element of $i^1(F)$ and so $[u]$ is indeed well-defined.
	We first need the following:
	\begin{lemma}\label{def_psi}
		The element $\eta \in \mathrm{J}^{-1}(F)$ and the symbols $[u]$ for $u \in F^\times$ defined above satisfy the Milnor-Witt relations,
		{\normalfont{i.e}} there exists a morphism $\psi : \KMW_*(F) \to \mathrm{J}^*(F)$ such that $\psi(\eta) = \eta$ and $\psi([u]) = [u]$ for all $u \in F^\times$.
	\end{lemma}
	\begin{proof}
		Let us check that the four relations defining $\KMW_*(F)$ hold in the graded ring $\mathrm{J}^\ast(F)$ with the given symbols $\eta$ and $[u]$ so that we indeed get a morphism $\psi: \KMW_*(F) \to \mathrm{J}^*(F)$.
		The Steinberg relation~\ref{rel1} holds by definition for symbols $\{u\}$ of $\KM_*(F)$. By~\ref{GW3} and~\ref{GW2} one has \begin{align*}
			\left(\langle u \rangle + \langle -1 \rangle\right) \otimes \left(\langle 1 - u \rangle + \langle - 1 \rangle\right) & = \langle u(1-u) \rangle + \langle -u \rangle + \langle u - 1 \rangle + 1  \\
			                                                                                                                     & = \langle u(1-u) \rangle + \langle - 1 \rangle + \langle u(u-1)\rangle  +1 \\
			                                                                                                                     & = 0
		\end{align*}
		in $\mathrm{W}(F)$,
		so relation~\ref{rel1} holds in $\mathrm{J}^\ast(F)$ for symbols of the form $[u]$. By definition, $\{ab\} = \{a\} + \{b\}$ in $\KM_1(F)$, and $\eta.([a][b])$ corresponds in $\KM_1(F) \times \mathrm{I}(F)$ to $(0, (\langle a \rangle - 1)(\langle b \rangle - 1))$.
		In $\mathrm{I}(F)$, one gets \begin{align*}
			(\langle a \rangle - 1)\otimes(\langle b \rangle - 1) & = \langle ab \rangle - \langle a \rangle - \langle b \rangle + 1.
		\end{align*}
		So there are equalities \begin{align*}
			[a] + [b] + \eta [ab] & = (\{a\} + \{b\}, \langle a \rangle - 1 + \langle b \rangle - 1 + \langle ab \rangle - \langle a \rangle - \langle b \rangle + 1) \\
			                      & = (\{a\} + \{b\}, \langle ab \rangle - 1)                                                                                         \\
			                      & = [ab]
		\end{align*}
		in $\mathrm{J}^1(F)$
		hence relation~\ref{rel2} holds.
		Relation~\ref{rel3} is obvious due to our choice of $\eta$ as the unit element of $\mathrm{W}(F)$.
		Finally, relation~\ref{rel4} holds since \begin{align*}
			h & = \eta[-1] + 2                    \\
			  & = (0, \langle -1 \rangle - 1) + 2 \\
			  & = (1, \langle - 1 \rangle + 1)    \\
			  & = (1, 0),
		\end{align*}
		by~\ref{GW2}, so that $\eta.h = (0, 0)$.
	\end{proof}
	To conclude the proof of Theorem~\ref{prop_iso}, we will prove:
	\begin{lemma}\label{psi_surj}
		The map $\psi_*: \KMW_*(F) \to \mathrm{J}^*(F)$ defined as above is surjective in all degree.
	\end{lemma}
	\begin{proof}
		First, notice that given any class $v \in I$, there exists some $u \in F^\times$ such that $(\{u\}, v) \in \mathrm{J}^1(F)$. Indeed, the congruence relations \begin{align*}
			\langle \alpha \rangle + \langle \beta \rangle \equiv \langle -\alpha\beta \rangle + \langle -1 \rangle & \mod \mathrm{I}^2(F) \\
			\text{and} \quad                                                                                                               \\
			\langle -1 \rangle + \langle -1\rangle + \langle -1 \rangle \equiv \langle 1 \rangle                    & \mod \mathrm{I}^2(F)
		\end{align*}
		in $\mathrm{W}(F)$ imply that any class $v = \langle \alpha_1 \rangle + \cdots \langle \alpha_{2r} \rangle$ is congruent modulo $\mathrm{I}^2(F)$ to some $\langle u \rangle + \langle -1\rangle$.
		This precisely means that $(\{u\}, v) \in \mathrm{J}^{1}(F)$.
		Moreover, any two $u_1, u_2$ such that $(\{u_1\},v) \in \mathrm{J}^{1}(F)$ and $(\{u_2\}, v) \in \mathrm{J}^{1}(F)$ must be such that $u_1$ and $u_2$ only differ by a square in $F^{\times}$ as
		$s_1: \KM_1(F) \to \mathrm{I}(F)/\mathrm{I}^2(F)$
		has a kernel precisely equal to $2\KM_{1}(F)$ so that $s_1$ induces an isomorphism $F^{\times}/{(F^{\times})}^2 \similarrightarrow \mathrm{I}(F)/\mathrm{I}^2(F)$.
		We claim that if $(\{u\}, v)$ is in the image of $\psi$, then so is $\left(\left\{u'\right\}, v\right)$ for any $u'$ such that $\left(\left\{u'\right\},v\right) \in \mathrm{J}^{1}(F)$.
		Indeed, one can write $u = u'w^2$, if $(\{u\},v) = \sum\limits_{i=1}^r\left(\left[\alpha_i\right], \langle \alpha_i \rangle + \langle -1 \rangle\right)$, then $u = \prod\limits_{i=1}^{r}\alpha_i$.
		Replacing $\alpha_1$ by $\alpha_1w^{-2}$ does not affect the Witt class $v$ and gives \begin{align*}
			\left(\left\{u'\right\}, v\right) = \left(\left\{\alpha_1v^{-2}\right\}, \langle \alpha_1 \rangle + \langle -1 \rangle\right) + \sum\limits_{i=2}^r(\{\alpha_i\}, \langle \alpha_i \rangle + \langle -1 \rangle).
		\end{align*}
		Thus, to show surjectivity on $\mathrm{J}^1(F)$, it suffices to show that for every $v \in I$, there is some $u \in F^\times$ such that
		$(\{u\}, v) \in \mathrm{J}^1(F)$ and $(\{u\}, v)$ is in the image of $\KMW_1(F) \to \mathrm{J}^1(F)$.
		This is certainly true if $v = \langle \alpha \rangle - 1$ for some $\alpha$ and it is known that such classes generate additively the group $\mathrm{I}(F)$. So $\KMW_1(F) \to \mathrm{J}^1(F)$ is an epimorphism.

		Now let $(u, v) \in \mathrm{J}^n(F)$ and recall that $\KM_\ast(F)$ is generated as a ring by $\KM_1(F)$, so that
		$u = \sum\limits_{i=1}^r\prod\limits_{j=1}^n\left\{a_{ij}\right\}$
		with $a_{i,j} \in F^\times$.
		By definition, $v \equiv s_n(u)\mod\mathrm{I}^{n+1}$, so
		\begin{align*}
			v & = \sum\limits_{i=1}^r\bigotimes\limits_{j=1}^n\langle 1, -a_{ij} \rangle + c
		\end{align*}
		with $c \in \mathrm{I}^{n+1}$. This yields
		\begin{align*}
			(u,v) = \sum\limits_{i=1}^r\prod\limits_{j=1}^n \left(\{a_{ij}\}, \langle a_{ij} \rangle + \langle -1 \rangle\right) + (0,c).
		\end{align*}
		Notice that since $c \in \mathrm{I}^{n+1}$, $(0,c) \in \mathrm{J}^{1}(F) = \mathrm{Im}(\psi_1)$ which proves that $\psi_n$ is indeed surjective if $n \geq 1$ as we already know that $\psi_1$ is surjective.

		It remains to show surjectivity if $n \leq 0$. For $n = 0$. The pullback square defining $\mathrm{J}(F)$ is in fact
		\[\begin{tikzcd}
				\GW(F) & \mathrm{W}(F) \\
				\mathbb{Z} & \mathbb{Z}/2 \mathbb{Z}
				\arrow["{}", from=1-1, to=1-2]
				\arrow["{}"', from=1-1, to=2-1]
				\arrow["{}", from=1-2, to=2-2]
				\arrow["{}"', from=2-1, to=2-2]
				\arrow["\ulcorner"{anchor=center, pos=0.125}, draw=none, from=1-1, to=2-2]
			\end{tikzcd}\]
		and the symbol $\langle u \rangle$ of $\KMW_0(F)$ is sent to $\langle u \rangle$ of $\GW(F) = \mathrm{J}^{1}(F)$.
		Indeed $\langle u \rangle = 1 + \eta[u]$, so $\langle u \rangle$ corresponds to the element $\eta.([u], \langle u \rangle - 1) + (1, 1)$ which is the element $(1, \langle u \rangle)$ since multiplication by $\eta$ sends a class $(a, b) \in \mathrm{J}^{k}(F)$ to $(0, b)$, viewed as an element of $\mathrm{J}^{k-1}(F)$.
		The element $(1, \langle u \rangle)$ corresponds to $\langle u \rangle$ with the identification $\mathrm{J}^{0}(F) = \GW(F)$.
		For $n < 0$, the pullback square is again the square
		\[\begin{tikzcd}
				\mathrm{W}(F) & \mathrm{W}(F) \\
				0 & 0
				\arrow["{}", from=1-1, to=1-2]
				\arrow["{}"', from=1-1, to=2-1]
				\arrow["{}", from=1-2, to=2-2]
				\arrow["{}"', from=2-1, to=2-2]
				\arrow["\ulcorner"{anchor=center, pos=0.125}, draw=none, from=1-1, to=2-2]
			\end{tikzcd}\]
		The element $\langle u \rangle$ of $\mathrm{J}^{n}(F) = \mathrm{W}(F)$ corresponds to $(0, \langle u \rangle)$ with this identification. This element is exactly $\eta^{-n}(1, \langle u \rangle)$, \textit{i.e} $\eta^{-n}\psi(\langle u \rangle)$.
	\end{proof}
	Back to the proof of Theorem~\ref{prop_iso}. We know that $\psi$ is an epimorphism, and for $n \geq 0$, the composite $\phi_n \circ \psi_n $ is the identity. Indeed, the proof of the preceding lemma shows that this is the case on generators. But since both maps are epimorphism, this means that $\phi_n$ is an isomorphism.
\end{proof}
\begin{corollarynoproof}\label{cor_loc}
	The canonical morphism $\KMW_*(F) \to \mathrm{W}(F)\left[t,t^{-1}\right]$ defined by \[[u] \to t^{-1}(\langle u \rangle - 1)\] induces an isomorphism $\KMW_*(F)\left[\eta^{-1}\right] \to \mathrm{W}(F)\left[t,t^{-1}\right]$.
\end{corollarynoproof}
\begin{lemma}\label{lemma_puissances}
	Let $a \in F^\times$ and $n \in \mathbb{Z}$. Let $n_{\epsilon} \colonequals \sum\limits_{i=1}^n\left\langle {(-1)}^{i-1} \right\rangle$ if $n \geq 0$ and $-\langle-1\rangle{(-n)}_\epsilon$ if $n < 0$.\\
	Then, we have $[a^n] = n_\epsilon [a]$ in $\KMW_1(F)$.
\end{lemma}
\begin{proof}
	We proceed by induction: if $n = 0$, this is clear: both expressions are $0$. For $n \geq 1$,
	\begin{align*}
		\left[a^n\right] & = \left[a^{n-1}\right] + [a] + \eta\left[a^{n-1}\right][a]          \\
		                 & = {(n-1)}_\epsilon[a] + [a] + {(n-1)}_\epsilon\eta[a][a]            \\
		                 & = {(n-1)}_\epsilon[a] + [a] + {(n-1)}_\epsilon\eta[-1][a]           \\
		                 & = {\left({(n-1)}_\epsilon + 1 + {(n-1)}_\epsilon\eta[-1]\right)}[a] \\
		                 & = \left({(n-1)}_\epsilon\left(1 + \eta[-1]\right) + 1\right)[a]     \\
		                 & = \left(\langle-1\rangle {(n-1)}_{\epsilon} + 1\right)[a]           \\
		                 & = n_\epsilon[a].
	\end{align*}
\end{proof}
\begin{proposition}
	Let $F$ be a field in which any unit is a square. Then the epimorphism \[\KMW_*(F) \to \KM_*(F)\] is an epimorphism in degree $\geq 0$, and $\KMW_*(F) \to \KW_*(F)$ is an isomorphism in degree < 0.
\end{proposition}
\begin{proof}
	Since $-1$ is a square. We have that $\langle -1 \rangle = 1$, so that $h = 2$ and relation~\ref{rel4} becomes $2\eta = 0$. By the previous lemma, $\eta[a^2] = 2\eta[a]$, so $\eta[a^2] = 0$. Since any unit is a square, $\eta[b] = 0$ for all $b$, so that relation~\ref{rel2} becomes $[ab] = [a] + [b]$.
	Thus in degree $\geq 0$, $\KM_*$ and $\KMW_*$ have the same generator and relations.
	In degree $< 0$, we know that $\KMW_n(F) \cong \mathrm{W}(F)$ by Proposition~\ref{prop_iso}. Since $h = 2$, $K^{W}_n(F) \cong \mathrm{W}(F)/2\mathrm{W}(F)$. But it is known that for a field in which every unit is a square, $\mathrm{W}(F) \cong \mathbb{Z}/2\mathbb{Z}$.
	So that taking the quotient by $2$ does not change the Witt Ring.
\end{proof}
\section{Witt K-theory in characteristic 2}
In this section, we bring a particular attention to Witt K-theory in characteristic two. Recall that $\KW_*(F)$ is defined to be the quotient $\KMW_*(F)/(h)$.
Recall also that we set $\mathrm{I}^n(F) \colonequals \mathrm{W}(F)$ for $n \leq 0$ and that $\mathrm{I}^n(F)$ is the usual $n$-th power of the fundamental ideal if $n > 0$.
With these notations, $\mathrm{I}^*(F) \colonequals \bigoplus\limits_{n \in \mathbb{Z}} \mathrm{I}^n(F)$ is a graded $\mathrm{W}(F)$-algebra.

Given an element $a \in F^\times$, the class of $[a]$ in $\KW_*(F)$ will be denoted by $\lBrack a \rBrack$. We will denote by $\lAngle a \rAngle \colonequals 1 + \langle -a \rangle \in \mathrm{W}(F)$ the Pfister form attached to $a$.
Notice that $\lAngle a \rAngle = 1 - \langle a \rangle$ in $\mathrm{W}(F)$, as $-\langle a \rangle = \langle -a \rangle$ by~\ref{GW2}.
Although the signs are important when working in a general setting to have well-defined morphisms, we will mostly ignore them when working in a characteristic $2$ setting, leaving them only to indicate how things might go in a more general setting.
Our goal is to prove the following:
\begin{theorem}\label{KW_iso_I}
	Let $F$ be a field of characteristic $2$, there is a unique isomorphism $\KW_*(F) \to \mathrm{I}^*(F)$ sending the class of $\lBrack a \rBrack$ to the Pfister form $\lAngle a \rAngle$.
\end{theorem}

This theorem appears as the main theorem of~\cite{morel_puissances2004} in characteristic different from two.
The main interest of this result is that through the explicit generators and relations of $\KW_*(F)$, one gets explicit generators and relations for the powers of the fundamental ideal.
One of the main input is Milnor's conjecture that $s_{n,F}: \KM_n(F)/2\KM_n(F) \to \mathrm{I}^n(F)/\mathrm{I}^{n+1}(F)$ is an isomorphism of abelian groups, which has been proven in~\cite{OVV_exact_2007} (see also~\cite{rondigs_slices_2016}).
The characteristic 2 case has been claimed to be provable using similar methods. One can indeed use Arason and Baeza's presentation of the powers of the fundamental ideal in characteristic 2~\cite{arason_baeza_2007} as done in~\cite{omanovic_2015}.
In the characteristic 2 case, the input is again that $s_{n,F}$ is an isomorphism, however it should be noted that this is a much more elementary fact in characteristic 2:
an early proof of this case was given by Kato in~\cite{kato_symmetric_1982} and the techniques of proof are rather elementary compared with the ones involved in the general case.
The proof we give is rather direct and completely elementary except for Kato's theorem.

In this section, $F$ will denote a field of characteristic two unless explicitly stated otherwise.
\begin{lemma}\label{simpl_1}
	Let $F$ be a field of characteristic 2. The following assertions hold in $\KMW_*(F)$:\begin{enumerate}[label = (\roman*)]
		\item $h = 2$,
		\item $\epsilon = -1$,
		\item $\forall n \in \mathbb{Z},\, n_\epsilon = n$,
		\item $\forall a \in F^\times,\, {[a]}^2 = 0$.
	\end{enumerate}
\end{lemma}
\begin{proof}
	The first point is the definition of $h = \eta[-1] + 2$ with the fact that $[-1] = [1]$, which is zero by lemma 3. The second point is that $-\langle -1 \rangle = - \langle 1 \rangle$, which is $-1$ by lemma 3.
	The third point is the fact that $n_\epsilon$ becomes a sum of $n$ copies of $\langle 1 \rangle = 1$, and the last point is the fact that
	\begin{align*}
		{[a]}^2 & = [a][-a] \\
		        & = 0,
	\end{align*}
	the last equality being the first point of lemma 5.
\end{proof}
\begin{lemma}
	Let $F$ be a field of characteristic 2, then $\KW_*(F)$ is commutative.
\end{lemma}
\begin{proof}
	Since $\epsilon = -1$ in $\KMW_*(F)$ and $1 = -1$ in $\KW_*(F)$ (as $h = 2$), $\epsilon$-commutativity of $\KMW_*(F)$ yields the lemma.
\end{proof}
\begin{lemma}\label{square_in_brack}
	Let $F$ be a field of characteristic 2. For all $a \in F^\times$, one has $\lBrack a^2 \rBrack = 0$ in $\KW_*(F)$.
\end{lemma}
\begin{proof}
	As $[a^2] = n_\epsilon[a]$ in $\KMW_*(F)$, one gets that $\lBrack a^2 \rBrack = 2\lBrack a \rBrack$ in $\KW_*(F)$ by Lemma~\ref{simpl_1}, which is also $0$ as $h = 2$ in $\KMW_*(F)$.
\end{proof}
The following fact was noticed by Morel in~\cite{morel_puissances2004}.
\begin{proposition}\label{KW_to_I}
	Let $F$ be an arbitrary field (of any characteristic). There is a unique map of graded $\mathrm{W}(F)$-algebras $\theta : \KW_*(F) \to \mathrm{I}^*(F)$ that sends $\lBrack a \rBrack$ to $\lAngle a \rAngle \in \mathrm{I}^1(F)$ and $\eta$ to $-1 \in \mathrm{I}^{-1}(F) = \mathrm{W} (F)$.
\end{proposition}
\begin{proof}
	In the ring $\mathrm{W}(F)$, one can compute that \begin{align*}
		(1 + \langle -a \rangle ) + (1 + \langle -b \rangle) - (1 + \langle a \rangle)(1 + \langle b \rangle) & = 1 - \langle a b \rangle
	\end{align*}
	which shows that $\lAngle ab \rAngle = \lAngle a \rAngle + \lAngle b \rAngle - \lAngle a \rAngle \lAngle b \rAngle$ in $\mathrm{I}^*(F)$. Similarly, one sees that \begin{align*}
		(1 - \langle a \rAngle)(1 - \langle 1 - a \rangle) & = 1 - \langle a\rangle - \langle 1 - a \rangle + \langle a(1-a) \rangle \\
		                                                   & = 1 - (1 + \langle a(1-a) \rangle) + \langle a (1-a) \rangle            \\
		                                                   & = 0
	\end{align*}
	the second equality being axiom~\ref{GW3}.
	Similarly, one checks that $-\lAngle -1 \rAngle = -h$, which is $0$ in $\mathrm{W}(F)$.
	This shows that $a \mapsto \lAngle a \rAngle \in \mathrm{I}^1(F)$ and $\eta \mapsto -1 \in \mathrm{I}^{-1}(F) = \mathrm{W}(F)$ extends to a map of graded ring as desired.
	Compatibility with the $\mathrm{W}(F)$-algebras structure then comes from the fact that $\langle a \rangle \cdot 1$ is sent to the element $1 - \lAngle a \rAngle = \langle a \rangle$ of $\mathrm{W}(F)$ as $\langle a \rangle = 1 + \eta \lBrack a \rBrack$ in $KW_*(F)$.
\end{proof}

\begin{proposition}\label{GW_relations_in_KW}
	Let $a, b \in F^\times$. \begin{enumerate}[label= (\roman*)]
		\item One has $\lBrack a^2 b \rBrack = \lBrack b \rBrack$ in $\KW_*(F)$.
		\item One has $\lBrack a \rBrack + \lBrack -a \rBrack = \lBrack 1 \rBrack + \lBrack -1 \rBrack$ in $\KW_*(F)$.
		\item If $a + b \in F^\times$, then $\lBrack a \rBrack + \lBrack b \rBrack = \lBrack a+b \rBrack + \lBrack ab(a+b) \rBrack$ in $\KW_*(F)$.
	\end{enumerate}
\end{proposition}
\begin{proof}
	Notice that the second point is trivial, as the left hand side is $2\lBrack a \rBrack$, which is zero as $2 = 0$ in $\KW_*(F)$, and the right hand side is also zero by Lemma~\ref{simpl_1}.

	For the first point, one has \begin{align*}
		\left\lBrack a^2b \right\rBrack = \left\lBrack a^2 \right\rBrack + \lBrack b \rBrack + \eta \left\lBrack a^2 \right\rBrack \lBrack b \rBrack
	\end{align*}
	and the claim exactly follows from Lemma~\ref{square_in_brack}.

	For the last point, first, setting $\alpha = \frac{a}{a + b}$, one sees that $1 - \alpha = \frac{b}{a + b}$, so that $\alpha(1-\alpha) = \frac{ab}{{(a+b)}^2}$, which differs multiplicatively from $ab$ by a square. Hence by the first point,
	\[\lBrack \alpha(1-\alpha) \rBrack = \lBrack ab \rBrack.\]
	On the other hand \[\lBrack \alpha(1-\alpha)\rBrack = \lBrack \alpha \rBrack + \lBrack 1 - \alpha \rBrack + \eta \lBrack \alpha \rBrack \lBrack 1 - \alpha\rBrack.\]
	The last summand of the right hand side is zero by the Steinberg relation, so that
	\[\lBrack ab \rBrack = \left\lBrack \frac{a}{a+b} \right\rBrack + \left\lBrack \frac{b}{a+b} \right\rBrack.\]
	Since one has
	\[ \left\lBrack \frac{a}{a+b} \right\rBrack = \lBrack a \rBrack - \left \langle \frac{a}{a+b} \right \rangle \lBrack a + b \rBrack\]
	and
	\[ \left \lBrack \frac{b}{a+b} \right \rBrack = \lBrack b \rBrack - \left \langle \frac{b}{a+b} \right \rangle \lBrack a+b \rBrack\]
	by~\ref{lemma_1_5} of Lemma~\ref{lemma_1}, the above expression of $\lBrack ab \rBrack$ yields the equality \[
		\lBrack ab \rBrack = \lBrack a \rBrack + \lBrack b \rBrack - (\langle \alpha \rangle + \langle 1 - \alpha \rangle) \lBrack a + b \rBrack
	\]
	As $x \mapsto \langle x \rangle$ extends to a morphism $\mathrm{GW}(F) \to \KW_*(F)$ (as it already extends to a morphism $\mathrm{GW}(F) \to \KMW_*(F)$), one gets $\langle \alpha \rangle + \langle 1 - \alpha \rangle = \langle 1 \rangle + \langle \alpha(1 - \alpha) \rangle$, which is also $ 1 + \langle ab \rangle$ as $\alpha(1-\alpha)$ and $ab$ differ multiplicatively by a square so that by reorganizing the terms of the above equality, it becomes
	\[\lBrack a \rBrack + \lBrack b \rBrack = \lBrack a + b \rBrack + \lBrack ab \rBrack + \langle ab \rangle \lBrack a + b \rBrack\]
	and the last two summands sum to \[ \lBrack ab(a + b) \rBrack\] by point~\ref{lemma_1_1} of Lemma~\ref{lemma_1}, which gives exactly the expected equality.
\end{proof}
\begin{proposition}
	There is a morphism \[\mathrm{I}(F) \to \KW_*(F)\] that sends the class of a Pfister form $\lAngle a \rAngle$ to $\lBrack a \rBrack$.
\end{proposition}
\begin{proof}
	Proposition~\ref{GW_relations_in_KW} tells us that $a \mapsto \lBrack a \rBrack$ extends to a morphism $\mathrm{GW}(F) \to \KW_*(F)$. Furthermore, as $h = 0$ in $\KW_*(F)$, this map factors as a map $f: \mathrm{W}(F) \to \KW_*(F)$.
	There is a short exact sequence of abelian groups

	\[ 0 \xrightarrow{} \mathrm{I}(F) \xrightarrow{} \mathrm{W}(F) \xrightarrow{\mathrm{rk}\ \mathrm{mod}\ 2} \mathbb{Z}/2\mathbb{Z} \xrightarrow{} 0\]

	As $F$ is of characteristic $2$, the element $h$ of $\mathrm{GW}(F)$ becomes $2$, so that the unit of the ring structure of $\mathrm{W}(F)$ factors as a map $\mathbb{Z}/2\mathbb{Z} \to \mathrm{W}(F)$ that splits the previous exact sequence.
	Thus, the subgroup $\mathrm{I}(F)$ identifies as a quotient of $\mathrm{W}(F)$, namely the quotient by the subgroup spanned by $\langle 1 \rangle$. As $\lBrack 1 \rBrack = 0$ in $\KW_*(F)$, the map $f: \mathrm{W}(F) \to \KW_*(F)$ factors through $\mathrm{I}(F)$ in a way that, by definition, the class of the Pfister form $\lAngle a \rAngle = 1 + \langle a \rangle$ is mapped to $\lBrack a \rBrack$.
\end{proof}
\begin{corollary}
	There is an isomorphism $\KW_1(F) \simeq \mathrm{I}(F)$.
\end{corollary}
\begin{proof}
	The morphism $\mathrm{I}(F) \to \KW_1(F)$ defined in the last proposition is surjective as $\KW_1(F)$ is spanned linearly by the elements $\lBrack u \rBrack$ for $u \in F^\times$ according to Lemma~\ref{lemma_generators}.
	Essentially by construction, the degree $1$ part of the morphism $\theta$ from Proposition~\ref{KW_to_I} is a right inverse to this morphism and it is thus injective.
\end{proof}
We now engage the proof that $\theta$ is an isomorphism in all degrees. Our argument will rely on the two following lemmas.
\begin{lemma}
	Let $F$ be a field of characteristic $2$ and $a, b, c \in F^\times$. \begin{enumerate}[label= (\roman*)]
		\item One has $\lBrack a \rBrack \lBrack b \rBrack = \lBrack a \rBrack \lBrack ab \rBrack$ in $\KW_*(F)$.

		\item If there exists $u, v \in F$ such that $c = u^2 + v^2a$, then
		      \[
			      \lBrack a \rBrack \lBrack b \rBrack = \lBrack a \rBrack \lBrack bc \rBrack
		      \] in $\KW_*(F)$.

		\item If there exists $u, v \in F$ such that $c = u^2a + v^2 b$, then \[
			      \lBrack a \rBrack \lBrack b \rBrack = \lBrack ab \rBrack \lBrack c \rBrack
		      \] in $\KW_*(F)$.
	\end{enumerate}
\end{lemma}
\begin{proof}
	For the first point, compute that \begin{align*}
		\lBrack a \rBrack \lBrack ab \rBrack & = \lBrack a \rBrack \left(\lBrack a \rBrack + \lBrack b \rBrack + \eta \lBrack a \rBrack \lBrack b \rBrack \right) \\
		                                     & = \lBrack a \rBrack^2 + \lBrack a \rBrack \lBrack b \rBrack + \eta \lBrack a \rBrack^2 \lBrack b \rBrack           \\
		                                     & = \lBrack a \rBrack \lBrack b \rBrack,
	\end{align*}
	the last equality following from the fact that $\lBrack a \rBrack^2 = 0$ by Lemma~\ref{simpl_1}.

	For the second point, in the case $u = 0$, then as $\lBrack a \rBrack = \left\lBrack a v^2 \right\rBrack$ by Proposition~\ref{GW_relations_in_KW}, we are done by the previous computation. Assume now that $u \neq 0$.
	One has, \begin{align*}
		\lBrack a \rBrack \lBrack b c \rBrack & = \lBrack a \rBrack \left\lBrack b \left(u^2 + av^2 \right) \right\rBrack                                                                                 \\
		                                      & = \lBrack a \rBrack \left(\lBrack b \rBrack + \left\lBrack u^2 + v^2a \right\rBrack + \eta \lBrack b \rBrack \left\lBrack u^2 + v^2a \right\rBrack\right)
	\end{align*}
	but using again Proposition~\ref{GW_relations_in_KW}, one sees that \begin{align*}
		\lBrack a \rBrack \left\lBrack u^2 + v^2 a \right\rBrack & = \lBrack a \rBrack \left\lBrack u^2\left(1 + {\left(\frac{v}{u}\right)}^2a\right) \right\rBrack                          \\
		                                                         & = \lBrack a \rBrack \left\lBrack 1 + {\left(\frac{v}{u}\right)}^2 a \right\rBrack                                         \\
		                                                         & = \left\lBrack {\left(\frac{v}{u}\right)}^2 a \right\rBrack \left\lBrack 1 + {\left(\frac{v}{u}\right)}^2 a \right\rBrack \\
		                                                         & = 0
	\end{align*}
	the last equality coming from the Steinberg relation and the characteristic 2 assumption. Thus \[\lBrack a \rBrack \lBrack bc \rBrack = \lBrack a \rBrack \lBrack ab \rBrack\] which is indeed $\lBrack a \rBrack \lBrack b \rBrack$ by the first point.

	Finally, assume $c = u^2a + v^2b$. Then as $c = u^2a\left(1 + {\left(\frac{v}{ua}\right)}^2 ab\right)$, one gets \begin{align*}
		\lBrack c \rBrack \lBrack ab \rBrack & = \left( \left\lBrack u^2 a \right\rBrack + \left \lBrack \left(1 + {\left(\frac{v}{ua}\right)}^2ab\right) \right \rBrack \right)\lBrack ab \rBrack     \\
		                                     & = \left\lBrack u^2 a \right\rBrack \lBrack ab \rBrack+ \left \lBrack \left(1 + {\left(\frac{v}{ua}\right)}^2ab\right) \right \rBrack \lBrack ab \rBrack \\
	\end{align*}
	Using again Proposition~\ref{GW_relations_in_KW}, characteristic $2$ and the Steinberg relation, the second summand of the right hand side is zero while the first summand is $\lBrack a \rBrack \lBrack ab \rBrack$, so we are done using the first point.
\end{proof}
\begin{lemma}\label{switch}
	Let $a_1,\ldots,a_n$ be elements of $F^\times$ and let $b_1$ be an element in $F^\times$. Assume $b_1$ is a norm for the pure subspace of the inner product space
	\[\bigotimes\limits_{i=1}^n \left( \langle 1 \rangle \bot \left\langle a_i \right\rangle\right),\]
	then there exists $b_2,\ldots, b_n$ such that \[
		\prod\limits_{i=1}^n\left\lBrack a_i \right\rBrack = \prod\limits_{i=1}^n \left\lBrack b_i \right\rBrack
	\]
	in $\KW_*(F)$.
\end{lemma}
\begin{proof}
	Given the previous lemma, the proof is exactly the same as in~\cite[Lemm.~4.1.4]{scharlau_quadratic_1985}, so we refer the reader to there.
\end{proof}
The last piece we need is the following theorem, due to Kato~\cite{kato_symmetric_1982}
\begin{theoremnoproof}
	Let $F$ be a field of characteristic $2$. The Milnor map \[s_* : \KM_*(F)/2\KM_*(F) \to \overline{\mathrm{I}^*}(F) \colonequals \bigoplus\limits_{n \geq 0} \mathrm{I}^n(F)/\mathrm{I}^{n+1}(F)\] that sends the symbol $\{a\}$ to the class of $\lAngle a \rAngle$ is an isomorphism.
\end{theoremnoproof}
Finally we can prove the main theorem. The proof strategy is that described by Morel in~\cite{morel_puissances2004} and attributed to Arason and Elmann. Notice that perhaps unsurprisingly, one of our characteristic 2 input is the same as in~\cite{arason_baeza_2007} which also establishes presentations for $\mathrm{I}^n(F)$ in characteristic 2.
\begin{theorem}
	Let $F$ be a field of characteristic $2$. The map $\theta_* : \KW_*(F) \to \mathrm{I}^*(F)$ constructed in Proposition~\ref{KW_to_I} induces an isomorphism in positive degrees.
\end{theorem}
\begin{proof}
	It suffices to show the assertion when $F$ is finitely generated over its prime subfield $\mathbb{F}_2$.
	Indeed, any $x \in \KW_n(F)$ such that $\theta(x) = 0$ has a lift in $\mathbb{Z}[F^\times]$ by Lemma~\ref{lemma_generators}, this lifts involves only finitely many field elements.
	By definition of $\theta$, the fact that $\theta(x) = 0$ implies that this lift of $x$ in $\mathbb{Z}[F^\times]$ is an element of the subgroup spanned by the Witt relation, and one can thus write it as a sum of elements that generates the Witt relations.
	All of this only involves a finite number of elements of $F$ and thus actually happens in the free abelian group on the nonzero elements of the subfield of $F$ generated by $\mathbb{F}_2$ and these elements. Similarly for surjectivity, given any $x \in \mathrm{I}^n(F)$, surjectivity in the subfield generated by elements appearing in a lift of $x$ to $\mathbb{Z}[F^\times]$ is enough.

	Now, such a field is finitely generated over $\mathbb{F}_2$, and so the dimension of $F$ as a vector space over its subfield of squares is finite and is a power of $2$~\cite[Cor. 1 of Thm. 4, n°6 of § 1]{bourbaki_alg_V} and for such field, it is known~\cite[Thm.~III.5.10]{milnor_husemoller1973} that $\mathrm{I}^n(F)$ is zero for any $n$ such that $2^n > \dim_{F^2}(F)$.
	Let us first show that $\KW_n(F)$ is also zero if $n > {\dim_{F^2}(F)}$.
	Let $n$ be an integer satisfying the latter inequality and let $a_1,\ldots,a_n$ be elements of $F$.
	Again by Lemma~\ref{lemma_generators}, it suffices to show that the product $\prod\limits_{i=1}^n \lBrack a_i \rBrack$ is zero.
	Because of the dimension assertion, the family of all nonempty products of the $a_i$ is not linearly independent over $F^2$. This is precisely saying that the pure subspace of the Pfister inner product space
	\[\lAngle a_1,\ldots, a_n \rAngle \colonequals \bigotimes_{i = 1}^n \big(\langle 1 \rangle \bot \langle a_i \rangle\big)\]
	is isotropic, and so it represents $1$. Thus, by Lemma~\ref{switch} we are done as $\lBrack 1 \rBrack$ is zero. This at least establishes that $\theta$ is an isomorphism in high enough degrees, as both source and targets are zero. Now, assuming that $\theta_k$ is an isomorphism for $k > n$, one can prove that $\theta_n$ is an isomorphism. Indeed, consider the diagram
	\[\begin{tikzcd}
			{\KW_{n+1}(F)} & {\KW_{n}(F)} & {{\left(\KW_{*}(F)/\eta\right)}_n} & 0 \\
			{\mathrm{I}^{n+1}(F)} & {\mathrm{I}^{n}(F)} & {\mathrm{I}^{n}(F)/\mathrm{I}^{n+1}(F)} & 0
			\arrow["{\theta_n}", from=1-2, to=2-2]
			\arrow["{\theta_{n+1}}", from=1-1, to=2-1]
			\arrow[from=2-1, to=2-2]
			\arrow["{\times \eta}", from=1-1, to=1-2]
			\arrow[from=1-2, to=1-3]
			\arrow[from=2-2, to=2-3]
			\arrow["{\sigma_n}", from=1-3, to=2-3]
			\arrow[from=1-3, to=1-4]
			\arrow[from=2-3, to=2-4]
			\arrow[Rightarrow, no head, from=1-4, to=2-4]
		\end{tikzcd}\]
	which obviously commutes and which has exact rows.
	As $\KW_*(F) = \KMW_*(F)/h$, the ring $\KW_*(F)/\eta$ is also the quotient of $\KMW_*(F)/\eta$ by the class of $h$. But the quotient of $\KMW_*(F)$ by $\eta$ is none other than $\KM_*(F)$. As $h = 2$ in the characteristic $2$ setting, ${\left(\KW_*(F)/\eta\right)}_n$ is thus isomorphic to $\KM_n(F)/2$.
	Following through the definitions, it is easy to see that along this identification the vertical map $\sigma_n$ gets identified to Milnor's map $s_n$. Thus, the previous diagram gets identified with
	\[\begin{tikzcd}
			{\KW_{n+1}(F)} & {\KW_{n}(F)} & {{\left(\KM_{*}(F)/2\right)}_n} & 0 \\
			{\mathrm{I}^{n+1}(F)} & {\mathrm{I}^{n}(F)} & {\mathrm{I}^{n}(F)/\mathrm{I}^{n+1}(F)} & 0
			\arrow["{\theta_n}", from=1-2, to=2-2]
			\arrow["{\theta_{n+1}}", from=1-1, to=2-1]
			\arrow[from=2-1, to=2-2]
			\arrow["{\times \eta}", from=1-1, to=1-2]
			\arrow[from=1-2, to=1-3]
			\arrow[from=2-2, to=2-3]
			\arrow["{s_n}", from=1-3, to=2-3]
			\arrow[from=1-3, to=1-4]
			\arrow[from=2-3, to=2-4]
			\arrow[Rightarrow, no head, from=1-4, to=2-4]
		\end{tikzcd}\]
	and all of the vertical maps except \textit{a priori} $\theta_n$ are isomorphisms, but as rows of this diagram are exact, $\theta_n$ is also an isomorphism and we can conclude by descending induction.
\end{proof}

\printbibliography
\end{document}